\documentclass[12pt,a4paper]{amsart}
\usepackage{amssymb,amsmath}
\usepackage{hyperref}

\numberwithin{equation}{section}
\newtheorem{theorem}{Theorem}[section]
\newtheorem{lemma}[theorem]{Lemma}
\newtheorem{proposition}[theorem]{Proposition}
\newtheorem{corollary}[theorem]{Corollary}

\theoremstyle{definition}
\newtheorem{definition}[theorem]{Definition}

\newcommand\Ass{\operatorname{Ass}}

\newcommand\Tor{\operatorname{Tor}}
\newcommand\Hom{\operatorname{Hom}}

\newcommand\Ext{\operatorname{Ext}}

\newcommand\depth{\operatorname{depth}}

\newcommand\reg{\operatorname{reg}}

\begin{document}
\title[Binomial edge ideal of a complete bipartite graph]{Algebraic properties of the
binomial edge ideal of a complete bipartite graph}
\author{Peter Schenzel, Sohail Zafar}
\address{Martin-Luther-Universit\"at Halle-Wittenberg,
Institut f\"ur Informatik, D --- 06 099 Halle (Saale),
Germany}

\email{peter.schenzel@informatik.uni-halle.de}

\address{Abdus Salam School of Mathematical Sciences, GCU, Lahore Pakistan}
\email{sohailahmad04@gmail.com}
\thanks{This research was partially supported by the Higher Education
Commission, Pakistan}

\begin{abstract}
Let $J_G$ denote the binomial edge ideal of a connected undirected graph on $n$ vertices. This is the ideal
generated by the binomials $x_iy_j - x_jy_i, 1\leq i < j \leq n,$ in the polynomial ring $S= K[x_1,\ldots,x_n,y_1,\ldots,y_n]$
where $\{i,j\}$ is an edge of $G$. We study the arithmetic properties of $S/J_G$ for $G$, the complete bipartite
graph.  In particular we compute dimensions, depths, Castelnuovo-Mumford regularities, Hilbert functions
and multiplicities of them.  As main results we give an explicit description of the modules of deficiencies,
the duals of local cohomology modules, and prove the purity of the minimal free resolution of $S/J_G$.
\end{abstract}


\maketitle
\section{Introduction} The main intention of the present paper is the study of
the binomial edge ideal of the complete bipartite graph. Let $G$ denote a connected undirected
graph on $n$ vertices labeled by $[n] = \{1,2,\ldots,n\}$.  For an arbitrary field $K$ let
$S = K[x_1,\ldots,x_n,y_1,\ldots,y_n]$ denote the polynomial ring in $2n$ variables. To the graph
$G$ one can associate the binomial edge ideal $J_G \subset S$ generated by binomials
$x_iy_j-x_jy_i, i < j,$ such that $\{i,j\}$ is an edge of $G$. This is an extension of the edge ideal (generated
by the monomials) as it was studied for instance in \cite{Vi}.  The binomial edge ideals
of  a graph $G$ might have applications in algebraic statistics (see \cite{HKR}). By the work of
Herzog et al. (see \cite{HKR}) the minimal primary decomposition of $J_G$ is known. Besides of that not so much
is known about the arithmetic properties of $S/J_G$. If $G$ denotes the complete graph on $n$ vertices, then
$S/J_G$ is the coordinate ring of the Segre embedding $\mathbb{P}^1_K \times \mathbb{P}^n_K$. This
is a variety of minimal degree. Therefore $S/J_G$ is a Cohen-Macaulay ring with a linear
resolution.  In the paper Ene, Herzog and Hibi (see \cite{EHH}) they studied Cohen-Macaulayness property for some special classes of graphs. By view of the primary
decomposition of $J_G$ (see \cite{HKR}) it follows that $S/J_G$ is not so often a Cohen-Macaulay
ring. As a certain generalization of the Cohen-Macaulay property the second author has studied
approximately Cohen-Macaulay rings (see \cite{So}). In the present paper we investigate the
binomial edge ideal of another important class of graphs, namely the complete bipartite
graph $G = K_{m,n}$ (see the definitions in Section 3).

As the main result of our investigations we prove (among others) the following results:

\begin{theorem} With the previous notation let $J_G \subset S$ denote the binomial edge ideal
associated to the complete bipartite graph $G = K_{m,n}$.
\begin{itemize}
\item[(a)] $\dim S/J_G = \max \{ n+m+1, 2m\}$ and
 \begin{eqnarray*}
\depth S/J_G=\left\{\begin{array}{ll}
 m+2, & \hbox{if \, $n=1$,} \\
n+2, & \hbox{if \, $m\geq n>1.$}
 \end{array}\right.
\end{eqnarray*}
\item[(b)] There is an explicit expression of the Hilbert series and the multiplicity equals
\[
e(S/J_G) = \begin{cases} 1, & \textrm{ if  } m > n+1 \textrm{ or } n=1 \textrm{ and } m>2,\\
                                     2m,  & \textrm{ otherwise. }
                 \end{cases}
\]
\item[(c)] The Castelnuovo-Mumford regularity is $\reg S/J_G = 2$ and $S/J_G$ admits a pure minimal
free resolution.
\item[(d)] There are at most 5 non-vanishing local cohomology modules $H^i_{S_+}(S/J_G)$. The modules of
deficiencies $\omega^i(S/J_G) = \Hom_K(H^i_{S_+}(S/J_G),K)$ are either Cohen-Macaulay modules
or the direct sum of two Cohen-Macaulay modules.
\item[(e)] $S/J_G$ is a Cohen-Macaulay canonically ring in the sense of \cite{Sch2}.
\end{itemize}
\end{theorem}

For the details on the modules of deficiencies we refer to Section 4 of the paper.
This is - at least for us - the first time in the literature that there is a complete description of the
structure of the modules of deficiencies besides of sequentially Cohen-Macaulay rings or
Buchsbaum rings. Our analysis is based on the primary decomposition of $J_G$ as shown
in \cite{HKR}.

In Section 2 we start with preliminary and auxiliary results needed in the rest of the paper.
In particular we give a short overview on the modules of deficiencies. In Section 3 we
study some of the properties of the the binomial edge ideal $J_G \subset S$  associated
to a complete bipartite graph. In Section 4 we give a complete list of all the modules of deficiencies
of the complete bipartite graphs. In the final Section 5 we prove the purity of the minimal
free resolution of $S/J_G$. This is the heart of our investigations. It gives in a natural way
some non-Cohen-Macaulay rings with pure resolutions. We might relate our investigations
as a better understanding of general binomial edge ideals.

\section{Preliminaries and auxiliary results}
First of all we will introduce the notation used in the sequel. Moreover we
summarize a few auxiliary results that we need.

We denote by $G$ a connected undirected graph on $n$ vertices labeled by
$[n] = \{1,2,\ldots,n\}$. For an arbitrary field $K$ let $S = K[x_1,\dots,x_n,y_1,\dots,y_n]$
denote the polynomial ring in the $2n$ variables $x_1,\dots,x_n,y_1,\dots,y_n$. To the
graph $G$ one can associate an ideal $J_G \subset S$ generated by all binomials
$x_iy_j-x_jy_i$ for all $1 \leq i < j \leq n$ such that $\{i,j\}$ is an edge of $G$.
This construction was invented by Herzog et al. in \cite{HKR} and \cite{EHH}. At first let us
recall some of their definitions.

\begin{definition} \label{2.1} Fix the previous notation. For a set $T \subset [n]$
let $\tilde{G}_T$ denote the complete graph on the vertex set $T$. Moreover
let $G_{[n]\setminus{T}}$ denote the graph obtained by deleting all vertices
of $G$ that belong to $T$.

Let $c = c(T)$ denote the number of connected components of $G_{[n]\setminus{T}}$.
Let $G_1,\ldots,G_c$ denote the connected components of $G_{[n]\setminus{T}}$. Then define
\[
P_{T}(G)=(\cup _{i\in T}\{x_{i},y_{i}\},J_{\tilde{G}_{1}},\dots,J_{\tilde{G}%
_{C(T)}}),
\]
where $\tilde{G}_i, i =1,\ldots,c,$ denotes the complete graph on the vertex set
of the connected component $G_i, i = 1,\ldots,c$.
\end{definition}

The following result is important for the understanding of the binomial edge ideal
of $G$.

\begin{lemma} \label{2.2}
With the previous notation the following holds:
\begin{itemize}
\item[(a)] $P_{T}(G)\subset S$
is a prime ideal of height $n-c+|T|,$ where $|T|$ denotes the number of
elements of $T$.
\item[(b)] $J_{G}=\cap _{T \subseteq [n]}P_{T}(G).$
\item[(c)] $J_{G}\subset P_{T}(G)$ is a
minimal prime if and only if either  $T=\emptyset $ \ or $T\neq \emptyset $
and $c(T \setminus \{i\})<c(T)$  each $i \in T$.
\end{itemize}
\end{lemma}

\begin{proof}
For the proof we refer to \cite{HKR}.
\end{proof}

Therefore $J_{G}$ is the intersection of prime ideals. That is, $S/J_{G}$ is
a reduced ring. Moreover, we remark that $J_G$ is an ideal generated by quadrics
and therefore homogeneous, so that $S/J_{G}$ is a graded ring with natural grading induced
by the $\mathbb{N}$-grading of $S$. As a technical tool we shall need the following
result.

\begin{proposition} \label{2.3} Let $I \subset S$ denote an ideal. Let $\underline{f}
= f_1,\ldots,f_r$ denote an $S/I$-regular sequence. Then $\underline{f}S \cap I =
\underline{f}I$.
\end{proposition}

\begin{proof} It is easy to see that $\Tor_1^S(S/\underline{f}S, S/I)
\cong \underline{f}S \cap I/\underline{f}I$. Moreover
\[
\Tor_1^S(S/\underline{f}S, S/I) \cong H_1(\underline{f};S/I),
\]
where $H_i(\underline{f};S/I)$ denotes the Koszul homology of $\underline{f}$
with respect to $S/I$. But these homology modules vanish for $i >0$.
\end{proof}

Let $M$ denote a finitely generated graded $S$-module. In the sequel we shall use also
the local cohomology modules of $M$ with respect to $S_+,$ denoted by $H^i(M),
i \in \mathbb{Z}$. Note that they are graded Artinian $S$-modules. We refer to the
textbook of Brodmann and Sharp (see \cite{BS})
for the basics on it. In particular the Castelnuovo-Mumford regularity $\reg M$ of $M$
is defined as
\[
\reg(M):= \max \{e(H^{i}(M)) + i| \depth(M)\leq i \leq \dim(M)\},
\]
where $e(H^{i}(M))$ is the least integer $m$ such that, for all $k > m,$ the degree $k$ part
of the $i$-th local cohomology module of M is zero.
For our investigations we need the following definition.

\begin{definition} \label{2.4} Let $M$ denote a finitely generated graded
$S$-module and $d = \dim M$. For an integer $i \in \mathbb{Z}$ put
\[
\omega^i(M) = \Ext_S^{2n-i}(M,S(-2n))
\]
and call it the $i$-th module of deficiency. Moreover we define $\omega(M) = \omega^d(M)$
the canonical module of $M$. We write also $\omega_{2\times}(M)= \omega(\omega(M)).$
These modules have been introduced and studied
in \cite{Sch1}.
\end{definition}

Note that by the graded version of Local Duality (see e.g. \cite{BS}) there is the
natural graded isomorphism $\omega^i(M) \cong \Hom_K(H^i(M),K)$ for all $i \in \mathbb{Z}$.
 For a finitely generated graded $S$-module $M$ and an integer
$i \in \mathbb{N}$ we set
\[
(\Ass M)_i = \{\mathfrak{p} \in \Ass M | \dim S/\mathfrak{p} = i\}.
\]
In the following we shall summarize a few properties on the modules of
deficiencies.

\begin{proposition} \label{2.5} Let $M$ denote a finitely generated
graded $S$-module and $d = \dim M.$
\begin{itemize}
\item[(a)] $\dim \omega^i(M) \leq i$ and $\dim \omega^d(M) = d$.
\item[(b)] $(\Ass \omega^i(M))_i = (\Ass M)_i$ for all $0 \leq i \leq d$.
\item[(c)] $M$ satisfies the Serre condition $S_2$ if and only if $\dim \omega^i(M)
\leq i-2$ for all $0 \leq i < d$.
\item[(d)] There is a natural homomorphism $M \to \omega^d(\omega^d(M))$. It is
an isomorphism if and only if $M$ satisfies the Serre condition $S_2$.
\item[(e)] For a homogeneous ideal $I \subset S$ there is a natural isomorphism
$\omega^d(\omega^d(S/I)) \cong \Hom_S(\omega^d(S/I),\omega^d(S/I)), d = \dim S/I,$ and it admits
the structure of a commutative Noetherian ring, the $S_2$-fication of $S/I$.
\item[(f)] The natural map $S/I \to \Hom_S(\omega^d(S/I),\omega^d(S/I)), d = \dim S/I,$ sends
the unit element to the identity map. Therefore it is a ring homomorphism.
\end{itemize}
\end{proposition}

\begin{proof} The results are shown in \cite{Sch1} and \cite{Sch2}. The proofs
in the graded case follow the same line of arguments.
\end{proof}

A decreasing sequence $\{M_i\}_{0\leq i \leq d}$ of a $d$-dimensional $S$-module
$M$ is called dimension filtration of $M$, if
$M_i/M_{i-1}$ is either zero or of dimension $i$ for all $i = 0, \ldots,d,$ where
$M_{-1} = 0.$ It was shown (see \cite{Sch3}) that the dimension filtration exists and
is uniquely determined.

\begin{definition} \label{2.6} An $S$-module $M$ is called sequentially Cohen-Macaulay
if the dimension filtration $\{M_i\}_{0\leq i \leq d}$ has the property that
$M_i/M_{i-1}$ is either zero or an $i$-dimensional Cohen-Macaulay module for all
$i = 0,\dots,d,$ (see \cite{Sch3}).
Note that in \cite{Sch3} this notion was originally called Cohen-Macaulay filtered.
\end{definition}

Note that a sequentially Cohen-Macaulay $S$-module $M$ with $\depth M \geq
\dim M -1$ was studied by G\^{o}to (see \cite{Go}) under the name approximately Cohen-Macaulay.
For our purposes here we need the following characterization of sequentially
Cohen-Macaulay modules.

\begin{theorem} \label{2.7}  Let $M$ be a finitely
generated graded $S$-module with $d = \dim M.$ Then the following conditions are equivalent:
\begin{itemize}
\item[(i)] $M$ is a sequentially Cohen-Macaulay.
\item[(ii)] For all $0 \leq i < d$ the module of deficiency $\omega^i(M)$ is either zero or an $i$-dimensional Cohen-Macaulay module.
\item[(iii)] For all $0 \leq i \leq d$ the modules $\omega^i(M)$ are either zero or $i$-dimensional Cohen-Macaulay modules.
\end{itemize}
\end{theorem}

\begin{proof} In the case of a local ring admitting a dualizing complex
this result was shown in \cite[Theorem 5.5]{Sch3}. Similar arguments work
also in the case of a finitely generated graded $S$-module $M$. Note that the equivalence
of (i) and (iii) was announced (without proof) in \cite{St}.
\end{proof}

\section{Complete bipartite graphs}
A bipartite graph is a graph whose vertices can be divided into two disjoint sets $V_{1}$ and
$V_{2}$ such every edge of $G$ connects a vertex in $V_{1}$ to one in $V_{2}$. Now the complete
bipartite graph is a bipartite graph $G$ such that for any two vertices, $v_{1}\in
V_{1}$ and $v_{2}\in V_{2}$, $v_{1}v_{2}$ is an edge in $G$. If $|V_{1}|=n$ and $|V_{2}|=m$ then it is usually  denoted by $K_{n,m}$. To simplifying notations we denote it often by $G$.

\begin{definition} \label{3.0}
For a sequence of
integers $1 \leq i_1 < i_2 < \ldots < i_k \leq n+m$ let $I(i_1,i_2,\ldots,i_k)$
denote the ideal generated by the $2 \times 2$ minors of the matrix
\[
\begin{pmatrix}
x_{i_1} & x_{i_2} & \cdots & x_{i_k} \\
y_{i_1} & y_{i_2} & \cdots & y_{i_k}
\end{pmatrix}.
\]
Note that $I(i_1,i_2,\ldots,i_k)$ is the ideal of the complete graph on the
vertex set $\{i_1,i_2,\ldots,i_k\}$.
\end{definition}

Let $J_{G}$ be the binomial edge ideal of complete bipartite graph on $[n+m]$ vertices and $J_{\tilde{G}}$ be the binomial edge ideal of complete graph on $[n+m]$ vertices. We begin with a lemma concerning the dimension of $S/J_G$.

\begin{lemma} \label{3.1} Let $G = K_{m,n}, m \geq n,$ denote the complete bipartite
graph. Let $\tilde{G}$ denote the complete graph on $[n+m]$. Let
$A_n = (x_1,\ldots,x_n,y_1,\ldots,y_n)$ for $n \geq 1$ and
$B_m = (x_{n+1},\ldots,x_{n+m},y_{n+1},\ldots,y_{n+m})$ for $n \geq 2$ and
$B_m = S$ for $n = 1$.
\begin{itemize}
\item[(a)] $J_G = J_{\tilde{G}} \cap A_n \cap B_m$ is the minimal primary decomposition of $J_G$.
\item[(b)] $\dim S/J_G = \max \{ n+m+1, 2m\}$.
\item[(c)] $(J_{\tilde{G}}\cap A_{n},J_{\tilde{G}}\cap B_{m})=J_{\tilde{G}}.$
\end{itemize}
\end{lemma}

\begin{proof} We start with the proof of (a). We use the statement proved in
Lemma \ref{2.2}. At first consider the case $m > n = 1$.
By view of Lemma \ref{2.2} we have to find all $ \emptyset \not= T \subseteq [1+m]$ such that
$c(T \setminus \{i\}) < c(T)$. Clearly $T_{0}=\{1\}$ satisfy the condition because $c(T_{0}) = m > 1$. Let $T$ denote $T \subset [1+m]$
a subset different of $T_{0}$. Then If $1\in T$ then $c(T)=m+1-|T|$ and $c(T\setminus \{i\})=m+2-|T|$ for $i\neq1$ and if $1\not\in T$ then $c(T)=1$ and
$c(T\setminus \{i\})=1$ for all $i\in T$. Hence we have the above primary decomposition.

Now consider the case of $m \geq n \geq 2$. As above we have to find all
$ \emptyset \not= T \subseteq [n+m]$ such that $c(T\setminus \{i\})<c(T)$ for all $i\in T$.
$T_{1}=\{1,2,\dots,n\}$ satisfy the above condition because $c(T)=m$ and $c(T\setminus \{i\})=1$
for all $i\in T$. Similarly $T_{2}=\{n+1,n+2,\dots,n+m\}$ also satisfies the above condition.

Our claim is that no other $T \subseteq [n+m]$ satisfies this condition. If $T_{1}\not\subseteq T$ and $T_{2}\not\subseteq T$ then $c(T)=1$ so in this case $T$ does not satisfy the above condition. Now suppose that $T_{1} \subsetneq T$ then $c(T)=m-|T\setminus T_{1}|$ and $c(T\setminus \{i\})=m+1-|T\setminus T_{1}|$ if $i \in T\setminus T_{1}$. The same argument works if $T_{2}\subsetneq T$. Hence we have $J_{G}=J_{\tilde{G}}\cap A_{n}\cap B_{m}$.

Then the statement on the dimension in (b) is a consequence of the reduced primary decomposition
shown in (a). To this end recall that $\dim S/A_n = 2m, \dim S/B_m = 2n$ and
$\dim S/J_{\tilde{G}} = n+m+1$.

For the proof of (c) we use the notation of the Definition \ref{3.0}.  Then it follows that
\[
J_{\tilde{G}}\cap A_{n} = (I(1,\ldots,n,n+i), i = 1,\ldots, m, I(n+1,\ldots,n+m)\cap A_n).
\]
Now $A_n$ consists of an $S/I(n+1,\ldots,n+m)$-regular sequence and
\[
I(n+1,\ldots,n+m)\cap A_n = A_n I(n+1,\ldots,n+m)
\]
by Proposition \ref{2.3}. Therefore we get
\[
J_{\tilde{G}}\cap A_{n} = (I(1,\ldots,n,n+i), i = 1,\ldots, m, A_n I(n+1,\ldots,n+m))
\]
and similarly
\[
J_{\tilde{G}}\cap B_{m} = (I(j,n+1,\ldots,n+m), j = 1,\ldots, n, B_m I(1,\ldots,n)).
\]
But this clearly implies that $(J_{\tilde{G}}\cap A_{n},J_{\tilde{G}}\cap B_{m})=J_{\tilde{G}}$ which proves the statement in (c).
\end{proof}

For the further computations we use the previous Lemma \ref{3.1}. In particular
we use three exact sequences shown in the next statement.

\begin{corollary} \label{3.2} With the previous notation we
have the following three exact sequences.
\begin{enumerate}
\item[(1)]
$0\to S/J_{G}\to S/J_{\tilde{G}}\cap A_{n}\oplus S/J_{\tilde{G}}\cap B_{m}\to S/J_{
\tilde{G}}\to 0$.
\item[(2)]
$0\to S/J_{\tilde{G}}\cap A_{n}\to S/J_{\tilde{G}}\oplus S/A_{n}\to S/(J_{
\tilde{G}},A_{n})\to 0$.
\item[(3)]
$0\to S/J_{\tilde{G}}\cap B_{m}\to S/J_{\tilde{G}}\oplus S/B_{m}\to S/(J_{%
\tilde{G}},B_{m})\to 0$.
\end{enumerate}
\end{corollary}

\begin{proof} The proof is an easy consequence of the primary decomposition
as shown in Lemma \ref{3.1}. We omit the details.
\end{proof}
Note that in case of $n=1$ we have $B_{m}=S$ therefore it is enough to consider the exact sequence $(2)$ as $(1)$ and $(3)$ gives no information.
\begin{corollary} \label{3.3} With the previous notation we have that
 \begin{eqnarray*}
\depth S/J_G=\left\{\begin{array}{ll}
 m+2, & \hbox{if \, $n=1$ ;} \\
n+2, & \hbox{if \, $m\geq n>1$}
 \end{array}\right.
\end{eqnarray*} and $\reg S/J_G \leq 2$.
\end{corollary}

\begin{proof} The statement is an easy consequence of the short exact sequences
shown in Corollary \ref{3.2}. To this end note that $S/J_{\tilde{G}},
S/(J_{\tilde{G}},A_n)$ and $S/(J_{\tilde{G}}, B_m)$ are Cohen-Macaulay rings
of dimension $n+m+1, m+1$ and $n+1$ respectively. Moreover $\reg S/J_{\tilde{G}} =
\reg S/(J_{\tilde{G}},A_n)= \reg S/(J_{\tilde{G}}, B_m) = 1$. By using the exact
sequences it provides the statement on the regularity. For the behaviour of
the depth respectively the regularity in short exact sequences see \cite[Proposition 1.2.9]{BH}
respectively \cite[Corollary 20.19]{Ei}.
\end{proof}

\section{The modules of deficiency}
The goal of this section is to describe all the local cohomology
modules $H^i(S/J_G)$ of the binomial edge ideal of a complete bipartite
graph $G$. We do this by describing their Matlis duals which by Local
Duality are the modules of deficiencies. Moreover, for a homogeneous
ideal $J \subset S$ let $H(S/J,t)$ denote the Hilbert series, i.e.
$H(S/J,t) = \sum_{i \geq 0} (\dim_K [S/J]_i) t^i$.

We start our investigations with the so-called star graph. That is complete
bipartite graph $K_{m,n}$ with $n = 1$. For $m \leq 2$ the ideal $J_G$ is a complete
intersection generated by one respectively two quadrics so let us assume that $m > 2$.

\begin{theorem} \label{4.1} Let $G$ denote the star graph $K_{m,1}$. Then the
binomial edge ideal $J_G \subset S$ has the following properties:
\begin{itemize}
\item[(a)] $\reg S/J_G = 2$.
\item[(b)] $\omega^i(S/J_G) = 0$ if and only if $i \not\in \{m+2,2m\}$.
\item[(c)] $\omega^{2m}(S/J_G) \cong S/A_1(-2m)$
\item[(d)] $\omega^{m+2}(S/J_G)$ is a $(m+2)$-dimensional Cohen-Macaulay module and there is an isomorphism $\omega^{m+2}(\omega^{m+2}(S/J_G))\cong (J_{\tilde{G}},A_1)/J_{\tilde{G}}$.
\end{itemize}
\end{theorem}

\begin{proof} We use the short exact sequence of Corollary \ref{3.2} (2). It induces
a short exact sequence
\[
0\to H^{m+1}(S/(J_{\tilde{G}},A_1))\to H^{m+2}(S/J_{G})\to H^{m+2}(S/J_{\tilde{G}}) \to 0
\]
and an isomorphism $H^{2m}(S/J_{G})\cong H^{2m}(S/A_1).$ Moreover the Cohen-Macaulayness of $S/J_{\tilde{G}}, S/A_1$ and $S/(J_{\tilde{G}},A_1)$ of dimensions
$m+2,$ $2m$ and $m+1$ respectively imply that $H^i(S/J_G) = 0$
if $i \not\in \{m+2,2m\}$.

The short exact sequence on local cohomology induces the following exact sequence
\[
0 \to \omega^{m+2}(S/J_{\tilde{G}}) \to \omega^{m+2}(S/J_G) \to \omega^{m+1}(S/(J_{\tilde{G}},A_1)) \to 0
\]
by Local Duality. Now we apply again local cohomology and take into account that
both $\omega^{m+2}(S/J_{\tilde{G}})$ and $\omega^{m+1}(S/(J_{\tilde{G}},A_1))$  are
Cohen-Macaulay modules of dimension $m+2$ and $m+1$ respectively. Then $\depth \omega^{m+2}(S/J_{G}) \geq m+1$.
By applying local cohomology and dualizing again it induces the following  exact
sequence
\[
0 \to \omega^{m+2}(\omega^{m+2}(S/J_G)) \to S/J_{\tilde{G}}\stackrel{f}{\to} S/(J_{\tilde{G}},A_1) \to
\omega^{m+1}(\omega^{m+2}(S/J_G)) \to 0.
\]
The homomorphism $f$ is induced by the commutative diagram
\[
\begin{array}{ccc}
	S/J_{\tilde{G}} & \to & S/(J_{\tilde{G}},A_1) \\
	\downarrow & & \downarrow \\
	\omega_{2\times}(S/J_{\tilde{G}}) & \to & \omega_{2\times}(S/J_{\tilde{G}},A_1).
\end{array}
\]
Note that the vertical maps are isomorphisms (see Proposition \ref{2.5}).
Since the upper horizontal map is surjective the lower horizontal map is surjective
too.
Therefore
$\omega^{m+1}(\omega^{m+2}(S/J_G)) = 0$. That is $\depth \omega^{m+2}(S/J_G) = m+2$ and it
is a Cohen-Macaulay module. Moreover $\omega^{m+2}(\omega^{m+2}(S/J_G)) \cong (J_{\tilde{G}},A_1)/J_{\tilde{G}}.$
This finally proves the statements in (b), (c) and (d).

It is well known that $\reg S/J_G = \reg S/(J_{\tilde{G}},A_1) = 1$ and $\reg S/A_1 = 0$. Then an inspection
with the short
exact sequence of Corollary \ref{3.2} shows that $\reg S/J_G = 2$.
\end{proof}

In the next result we will consider the modules of deficiencies
of the complete bipartite graph $G = K_{m,n}, n \geq 2$.

\begin{theorem} \label{4.2} Let $m \geq n > 1$ and assume that the pair
$(m,n)$ is different from $(n+1,n)$ and $(2n-2,n)$. Then
\begin{itemize}
\item[(a)] $\reg S/J_G = 2$.
\item[(b)] $\omega^i(S/J_G) = 0$ if and only if
$i \not\in \{n+2, m+2, 2n, m+n+1, 2m\}$ and there are the
following isomorphisms and integers
\[
\begin{array}{c|ccc}
i & \omega^i(S/J_G) & \depth \omega^i(S/J_G) & \dim \omega^i(S/J_G)\\
\hline
n+2 & \omega^{n+1}(S/(J_{\tilde{G}},B_m)) & n+1 & n+1 \\
m+2 & \omega^{m+1}(S/(J_{\tilde{G}},A_n)) & m+1 & m+1 \\
2n & S/B_m(-2n) & 2n & 2n \\
n+m+1 & \omega^{m+n+1}(S/J_{\tilde{G}}) & n+m+1 & n+m+1 \\
2m & S/A_n(-2m) & 2m & 2m
\end{array}
\]
\end{itemize}
\end{theorem}

\begin{proof}
Under the assumption of $n+1<m<2n-2$ it follows that $2m>m+n+1>2n>m+2>n+2$.
Then the short exact sequences (see Corollary \ref{3.2}) induce the following
isomorphisms:
\begin{enumerate}
\item[(1)]
$H^{n+2}(S/J_{G})\cong H^{n+2}(S/J_{\tilde{G}}\cap B_{m})\cong H^{n+1}(S/(J_{\tilde{G}},B_{m}))$,
\item[(2)] $H^{m+2}(S/J_{G})\cong H^{m+2}(S/J_{\tilde{G}}\cap A_{n})\cong H^{m+1}(S/(J_{\tilde{G}},A_{n}))$,
\item[(3)] $H^{2n}(S/J_{G})\cong H^{2n}(S/B_{m})$ and
\item[(4)] $H^{2m}(S/J_{G})\cong H^{2m}(S/A_{n})$.
\end{enumerate}
Moreover there is the following short exact sequence
\[
0\to H^{m+n+1}(S/J_{G})\to H^{m+n+1}(S/J_{\tilde{G}}\cap A_{n})\oplus H^{m+n+1}(S/J_{\tilde{G}}\cap B_{m})\to H^{m+n+1}(S/J_{\tilde{G}})\to 0
\]
and $H^i(S/J_G) = 0$ if $i \not\in \{n+2, m+2, 2n, m+n+1, 2m\}$.

Because of the short exact sequences in Corollary \ref{3.2} there are
isomorphisms
\[
H^{m+n+1}(S/J_{\tilde{G}}\cap B_{m}) \cong
H^{m+n+1}(S/J_{\tilde{G}}) \cong H^{m+n+1}(S/J_{\tilde{G}}\cap A_{n}).
\]
So by Local Duality we get a short exact sequence
\[
0 \to \omega(S/J_{\tilde{G}}) \to \omega(S/J_{\tilde{G}}) \oplus
\omega(S/J_{\tilde{G}}) \to \omega^{m+n+1}(S/J_G) \to 0.
\]
This implies  $\depth \omega(S/J_G) \geq n+m$. Moreover by applying local cohomology and again
the Local Duality we get the following commutative diagram with exact rows
\[
\begin{array}{ccccccccc}
0 \to  & S/J_G & \to & S/J_{\tilde{G}}\cap A_n \oplus S/J_{\tilde{G}} \cap B_m & \to &S/J_{\tilde{G}} & \to & 0 &   \\
 & \downarrow  &      &   \downarrow  &                                         &        \|      &   & \\
0  \to & \Omega(S/J_G) &  \to & S/J_{\tilde{G}} \oplus
S/J_{\tilde{G}} &  \stackrel{f}{\to} &  S/J_{\tilde{G}} & \to & \omega^{m+n}(\omega^{m+n+1}(S/J_G)) &  \to   0,
\end{array}
\]
where $\Omega(S/J_G) = \omega^{m+n+1}(\omega^{m+n+1}(S/J_G))$.
Now we show that $\omega^{m+n}(\omega^{m+n+1}(S/J_G)) = 0$. This follows since $f$ is easily seen to be surjective.
That is, $\omega^{m+n+1}(S/J_G)$ is a $(m+n+1)$-dimensional Cohen-Macaulay module. Moreover $f$ is a split
homomorphism and therefore $\Omega(S/J_G) \simeq S/J_{\tilde{G}}$ By duality this implies that
 $\omega^{m+n+1}(S/J_G) \cong \omega(S/J_{\tilde{G}})$. This completes the proof of the statements in (b).
 By similar arguments the other cases for $(m,n)$ different of $(n+1,n)$ and $(2n-2,n)$ can be proved. We omit the
 details.
 Clearly  $\reg S/J_G = 2$ as follows by (b). 
\end{proof}

As a next sample of our considerations let us consider the case of the complete bipartite graph $K_{m,n}$ with
$(m,n) = (n+1,n)$.

\begin{theorem} \label{4.3} Let $m = n+ 1$ and $n > 3$.  Then:
\begin{itemize}
\item[(a)] $\reg S/J_G = 2$.
\item[(b)] $\omega^i(S/J_G) = 0$ if and only if
$i \not\in \{n+2, n+3, 2n, 2n+2\}$ and there are the
following isomorphisms and integers
\[
\begin{array}{c|ccc}
i & \omega^i(S/J_G) & \depth \omega^i(S/J_G) & \dim \omega^i(S/J_G)\\
\hline
n+2 & \omega^{n+1}(S/(J_{\tilde{G}},B_{n+1})) & n+1 & n+1 \\
n+3 & \omega^{n+2}(S/(J_{\tilde{G}},A_n)) & n+2 & n+2 \\
2n & S/B_{n+1}(-2n) & 2n & 2n \\
2n+2 & \omega(S/J_{\tilde{G}}) \oplus S/A_n(-2n-2) & 2n+2 & 2n+2.
\end{array}
\]
\end{itemize}
\end{theorem}

\begin{proof}
 By  applying the local cohomology functors $H^{\cdot}(-)$ to the exact sequence (2) in Corollary \ref{3.2}  we get the
 following:
 \begin{itemize}
 \item[(1)] $H^{n+3}(S/J_{\tilde{G}}\cap A_{n})\cong H^{n+2}(S/(J_{\tilde{G}},A_{n})),$
 \item[(2)] $H^{2n+2}(S/J_{\tilde{G}}\cap A_{n})\cong H^{2n+2}(S/J_{\tilde{G}})\oplus H^{2n+2}(S/A_{n})$ and
\item[(3)] $H^{i}(S/J_{\tilde{G}}\cap A_{n})=0$ for $i \not= n+2, 2n+2$.
\end{itemize}
Similarly,  if we apply $H^{\cdot}(-)$ to the exact sequence (3) in Corollary \ref{3.2}  we get
\begin{itemize}
\item[(4)] $H^{n+2}(S/J_{\tilde{G}}\cap B_{n+1})\cong H^{n+1}(S/(J_{\tilde{G}},B_{n+1}))$.
 \item[(5)] $H^{2n}(S/J_{\tilde{G}}\cap B_{n+1})\cong H^{2n}(S/B_{n+1})$.
 \item[(6)] $H^{2n+2}(S/J_{\tilde{G}}\cap B_{n+1})\cong H^{2n+2}(S/J_{\tilde{G}})$.
 \item[(7)] $H^{i}(S/J_{\tilde{G}}\cap B_{n+1})=0 $ for $i \not= n+2,2n, 2n+2$.
\end{itemize}
With these results in mind the short exact sequence (1) of Corollary \ref{3.2} provides (by applying the local
cohomology functor) the vanishing $H^i(S/J_G) = 0$ for all $i \not= n+2, n+3,2n, 2n+2$. Moreover it induces
isomorphisms
\[
H^{n+2}(S/J_{G})\cong H^{n+1}(S/(J_{\tilde{G}},B_{n+1})) \textrm{ and }  H^{2n}(S/J_{G})\cong H^{2n}(S/B_{n+1})
\]
and as $n>3$ so $2n>n+3$ the isomorphism $H^{n+3}(S/J_{G})\cong  H^{n+2}(S/(J_{\tilde{G}},A_{n}))$. Moreover
we obtain the following short exact sequence
\[
0\to H^{2n+2}(S/J_{G})\to H^{2n+2}(S/J_{\tilde{G}})\oplus H^{2n+2}(S/A_{n})\oplus H^{2n+2}(S/J_{\tilde{G}})\to H^{2n+2}(S/J_{\tilde{G}})\to 0.
\]
By Local Duality this proves the first three rows in the table of statement (b). By Local Duality we get also the
the following short exact sequence
\[
0\to \omega(S/J_{\tilde{G}})\to \omega(S/J_{\tilde{G}})\oplus \omega(S/A_{n})\oplus \omega(S/J_{\tilde{G}})\to \omega(S/J_{G})\to 0.
\]
Note that we may write $\omega$ instead of $\omega^{2n+2}$ because all modules above are canonical modules.
First of all the short exact sequence provides that $\depth \omega(S/J_{G}) \geq 2n+1$.
By applying local cohomology and dualizing again  we get the following exact sequence
\[
0\to \omega_{2\times}(S/J_{G})\to S/J_{\tilde{G}}\oplus S/A_{n}\oplus S/J_{\tilde{G}} \stackrel{f}{\to} S/J_{\tilde{G}}\to
\omega^{2n+1}(\omega(S/J_G)) \to 0.
\]
As in the proof of Theorem \ref{4.1} we see that $f$ is surjective. Therefore
$\omega^{2n+1}(\omega(S/J_G)) =0$ and $\depth \omega(S/J_G) = 2n+2$. Whence $\omega(S/J_G)$ is a
$(2n+2)$-dimensional Cohen-Macaulay module. Then $f$ is a split surjection and
$\omega_{2\times}(S/J_{G})\cong S/J_{\tilde{G}}\oplus S/A_{n}$. This  implies the isomorphism  $\omega(S/J_{G})\cong \omega(S/J_{\tilde{G}})\oplus\omega(S/A_{n})$ and this finishes the proof of (b).
Clearly  $\reg S/J_G = 2$.
\end{proof}

\begin{theorem} \label{4.x} Let $m >3$ and $n =2$.  Then:
\begin{itemize}
\item[(a)] $\reg S/J_G = 2$.
\item[(b)] $\omega^i(S/J_G) = 0$ if and only if
$i \not\in \{4,m+2,m+3,2m\}$ and there are the
following isomorphisms and integers
\[
\begin{array}{c|ccc}
i & \omega^i(S/J_G) & \depth \omega^i(S/J_G) & \dim \omega^i(S/J_G)\\
\hline
4 & \omega^4((J_{\tilde{G}},B_m)/B_m) & 4 & 4 \\
m+2 & \omega^{m+1}(S/(J_{\tilde{G}},A_2)) & m+1 & m+1 \\
m+3 & \omega(S/J_{\tilde{G}}) & m+3 & m+3 \\
2m & S/A_2(-2m) & 2m & 2m.
\end{array}
\]
Moreover there is an isomorphism $\omega^{4}(\omega^{4}(S/J_G)\cong (J_{\tilde{G}},B_m)/B_m$.
\end{itemize}
\end{theorem}

\begin{proof}
 By  applying the local cohomology functors $H^{\cdot}(-)$ to the exact sequence (2) in Corollary \ref{3.2} we get the
 following:
 \begin{itemize}
 \item[(1)] $H^{m+2}(S/J_{\tilde{G}}\cap A_{2})\cong H^{m+1}(S/(J_{\tilde{G}},A_{2})),$
 \item[(2)] $H^{m+3}(S/J_{\tilde{G}}\cap A_{2})\cong H^{m+3}(S/J_{\tilde{G}})$,
\item[(3)] $H^{2m}(S/J_{\tilde{G}}\cap A_{2})\cong H^{2m}(S/A_{2})$ and
\item[(4)] $H^{i}(S/J_{\tilde{G}}\cap A_{2})=0$ for $i \not= m+2, m+3, 2m $.
\end{itemize}
Similarly,  if we apply $H^{\cdot}(-)$ to the exact sequence (3) in Corollary \ref{3.2}  we get the isomorphism $H^{m+3}(S/J_{\tilde{G}}\cap B_{m})\cong H^{m+3}(S/J_{\tilde{G}})$ and the exact sequence
\[
0\to H^{3}(S/(J_{\tilde{G}},B_m))\to H^{4}(S/J_{\tilde{G}}\cap B_m)\to H^{4}(S/B_m)\to 0.
\]
The short exact sequence on local cohomology induces the following exact sequence
\[
0 \to \omega^{4}(S/B_m) \to \omega^{4}(S/J_{\tilde{G}}\cap B_m) \to \omega^{3}(S/(J_{\tilde{G}},B_m)) \to 0
\]
by Local Duality. Now we apply again local cohomology and taking into account that
both $\omega^{4}(S/B_m)$ and $\omega^{3}(S/(J_{\tilde{G}},B_m))$  are
Cohen-Macaulay modules of dimension $4$ and $3$ respectively. Then $\depth \omega^{4}(S/J_{\tilde{G}}\cap B_m)) \geq 3$.
By applying local cohomology and dualizing again it induces the following  exact
sequence
\[
0 \to \omega^{4}(\omega^{4}(S/J_{\tilde{G}}\cap B_m)) \to S/B_m \stackrel{f}{\to} S/(J_{\tilde{G}},B_m) \to
\omega^{3}(\omega^{4}(S/J_{\tilde{G}}\cap B_m)) \to 0.
\]
Now the homomorphism $f$ is an epimorphism. 
$\omega^{3}(\omega^{4}(S/J_{\tilde{G}}\cap B_m)) = 0$. That is $\depth \omega^{4}(S/J_{\tilde{G}}\cap B_m) =4$ and it
is a Cohen-Macaulay module. Moreover $\omega^{4}(\omega^{4}(S/J_{\tilde{G}}\cap B_m)) \cong (J_{\tilde{G}},B_m)/B_m.$
With these results in mind the short exact sequence (1) of Corollary \ref{3.2} provides (by applying the local
cohomology functor) the vanishing $H^i(S/J_G) = 0$ for all $i \not= 4, m+2, m+3, 2m$. Moreover it induces
isomorphisms
\[
H^{4}(S/J_{G})\cong H^{4}(S/J_{\tilde{G}}\cap B_m), H^{m+2}(S/J_{G})\cong  H^{m+1}(S/(J_{\tilde{G}},A_{2}))
\]
and $H^{2m}(S/J_{G})\cong  H^{2m}(S/A_{2})$. Moreover
we obtain the following short exact sequence
\[
0\to H^{m+3}(S/J_{G})\to H^{m+3}(S/J_{\tilde{G}})\oplus H^{m+3}(S/J_{\tilde{G}})\to H^{m+3}(S/J_{\tilde{G}})\to 0.
\]
 This  implies the isomorphism  $\omega^{m+3}(S/J_{G})\cong \omega(S/J_{\tilde{G}})$.
\end{proof}
As a final step we shall consider the case of the complete bipartite graph $K_{m,n}$ with $2 n = m+2$.
In all of the previous examples we have the phenomenon that $\omega^i(S/J_G)$ is either zero or
a Cohen-Macaulay module with $i-1 \leq \dim\omega^i(S/J_G) \leq i$ for all $i \in \mathbb{Z}$. and the
canonical module $\omega(S/J_G) =\omega^d(S/J_G), d = \dim S/J_G,$ is a $d$-dimensional
Cohen-Macaulay module. For $2n = m+2$ this is no longer true.

\begin{theorem} \label{4.4} Let $m +2 = 2n$ and $m>n+1$.  Then:
\begin{itemize}
\item[(a)] $\reg S/J_G = 2$.
\item[(b)] $\omega^i(S/J_G) = 0$ if and only if
$i \not\in \{n+2, m+2 = 2n, m+n+1, 2m\}$ and there are the
following isomorphisms and integers
\[
\begin{array}{c|ccc}
i & \omega^i(S/J_G) & \depth \omega^i(S/J_G) & \dim \omega^i(S/J_G)\\
\hline
n+2 & \omega^{n+1}(S/(J_{\tilde{G}},B_m)) & n+1 & n+1 \\
m+2 & \omega^{m+1}(S/(J_{\tilde{G}},A_n)) \oplus S/B_m(-2n) & m+1& m+2 \\
m+n+1 & \omega(S/J_{\tilde{G}})  & m+n+1 & m+n+1 \\
2m & S/A_n(-2m) & 2m & 2m.
\end{array}
\]
\end{itemize}
\end{theorem}

\begin{proof}
It is easily seen that $n +2 < 2n = m+2 < m+n+1 < 2m$. Then the short exact sequences of Corollary
\ref{3.2} provide that $H^i(S/J_G) = 0$ for all $i \not= n+2, m+2 = 2n, m+n+1, 2m$. Moreover, it induces the
following isomorphisms
\begin{itemize}
\item[(1)] $H^{n+2}(S/J_G) \cong H^{n+1}(S/(J_{\tilde{G}},B_m))$,
\item[(2)] $H^{m+2}(S/J_G) \cong H^{m+1}(S/(J_{\tilde{G}},A_n) \oplus H^{m+2}(S/B_m)$,
\item[(3)] $H^{m+n+1}(S/J_G) \cong H^{m+n+1}(S/J_{\tilde{G}})$ and
\item[(4)] $H^{2m}(S/J_G) \cong H^{2m}(S/A_n)$.
\end{itemize}
This easily yields the statements in (a) and (b).
 \end{proof}

 The difference of the case handled in Theorem \ref{4.4} is the fact the $\omega^{m+2}(S/J_G)$ is not a
 Cohen-Macaulay module. It is the direct sum of two Cohen-Macaulay modules of dimensions
 $m+2$ and $m+1$ respectively.  In \cite{Sch2} a finitely generated $S$-module $M$ is called
 canonically Cohen-Macaulay module whenever $\omega(M)$ is a Cohen-Macaulay module. Note that if
 $M$ is a Cohen-Macaulay module, then it is also a Cohen-Macaulay canonical module. The converse
 is not true.

Now we prove some corollaries about the Cohen-Macaulayness and related properties.

\begin{corollary} \label{4.5}  Let $J_G \subset S$ denote the binomial edge ideal of a complete
 bipartite graph.
\begin{itemize}
\item[(a)] $S/J_G$ is a Cohen-Macaulay canonical ring and $\depth \omega^i(S/J_G)
 \geq i-1$ for all $\depth S/J_G \leq i \leq \dim S/J_G$. Moreover $S/J_G$ is a Cohen-Macaulay ring
 if and only if $(m,n) \in \{(2,1) (1,1)\}$.
 \item[(b)] $S/J_G$ is sequentially Cohen-Macaulay and not Cohen-Macaulay if and only if $n=1$ and $m>2$ or $n=m=2$.
 \end{itemize}
\end{corollary}

 \begin{proof} By view of Theorems \ref{4.1}, \ref{4.2},  \ref{4.3}, \ref{4.x}  and \ref{4.4} we get the statements on the
 Cohen-Macaulayness of $\omega(S/J_G)$ and the estimates of of the depth of $\omega^i(S/J_G)$
 for all possible bipartite graphs $G$.  By Lemma \ref{3.1} and Corollary \ref{3.3} the claim on the Cohen-Macaulayness of $S/J_G$ is easily seen. Similar arguments work for the sequentially Cohen-Macaulay property
 as it is easily seen by the definition.
  \end{proof}

\section{On the purity of the free resolution}

In the following let $J_r \subset S, r \leq m+n,$ denote the binomial edge ideal corresponding to the
complete graph on $r$ vertices. As a technical tool for our further investigations
we need the following Lemma.

\begin{lemma} \label{5.1} \textrm{(a)} Let $M$ denote a finitely generated graded $S$-module. Let $\underline{f}
= f_1,\ldots,f_l$ denote an $M$-regular sequence of forms of degree 1. Then
\[
\Tor_i^S(K, M/\underline{f}M) \cong \oplus_{j=0}^l\Tor_{i-j}^S(K,M)^{\binom{l}{j}}(-j).
\]
\textrm{(b)} $\Tor_i^S(K,S/J_r) \cong K^{b_i(r)}(-i-1)$ for $i = 1,\ldots,r-1,$ where $b_i(r) = i \binom{r}{i+1}$.
\end{lemma}

\begin{proof} For the proof of  the statement in (a) let $l = 1$ and $f = f_1$. Then the short exact sequence
$0 \to M(-1) \stackrel{f}{\to} M \to M/fM \to 0$ provides an isomorphism
\[
\Tor_i^S(K,M/fM) \cong \Tor_i^S(K,M) \oplus \Tor_{i-1}^S(K,M)(-1)
\]
for all $i \in \mathbb{Z}$. By an easy induction argument this yields the isomorphisms in (a). The statement
in (b) is well-known since $S/J_r$ has a linear resolution (see e.g. \cite[Exercise A2.19]{Ei}).
\end{proof}

For a certain technical reason we need the following Lemma that
describes the ideals $J_{\tilde{G}} \cap A_n$ respectively $J_{\tilde{G}} \cap B_m$ as binomial edge ideals.

\begin{lemma} \label{5.3} The ideal $J_{\tilde{G}} \cap A_n$ is the binomial edge ideal of the graph $G$ obtained by deleting all edges
$\{i,j\}$ of the complete graph on $[n+m]$ vertices such that
$n < i < j \leq m+n$. Similarly $J_{\tilde{G}} \cap B_m$ is the binomial edge ideal of the graph where all the edges $\{i,j\}$ of the complete graph on $[n+m]$ vertices such that $1 \leq i < j \leq n$ are deleted.
\end{lemma}

\begin{proof} Let us consider the ideal $J_{\tilde{G}} \cap A_n$. Look at the primary decomposition of the graph $G$. We have to find
all $\emptyset \not= T \subset [n+m]$ such that $c(T \setminus \{i\})
< c(T)$ for all $i \in T$. If $T = \{1,\ldots,n\}$, then $c(T) = m > 1$ and
$c(T \setminus \{i\}) = 1$ for all $i$. Let $T \subset [n+m]$ denote a subset with $T \not=
\{1,\ldots,n\}$. Then it is easy to see that the condition
$c(T \setminus \{i\})<c(T)$ for all $i \in T$ can not be satisfied. So the
claim follows by Lemma \ref{2.2}. A similar consideration proves the
case of $J_{\tilde{G}} \cap B_m$.
\end{proof}

As usual we define $\beta_{i,j}(M) = \dim_K \Tor_i^S(K,M)_{i+j}, i,j \in \mathbb{Z},$ the graded Betti
numbers of $M$, a finitely generated $S$-module. Then $\reg M = \max \{j \in \mathbb{Z} | \beta_{i,j}(M) \not= 0\}$.
In the following we shall prove that $S/J_G$ has a pure resolution. Note that all the $\beta_{i,j}(S/J_G)$ outside
of the Betti table are zero.

\begin{theorem} \label{5.2} Let $S/J_G$ denote the binomial edge ideal of the complete bipartite
graph $K_{m,n}$. Then the Betti diagram has the following form
\[
\begin{array}{c|ccccc}
   & 0 & 1 & 2 & \cdots & p \\
\hline
0 & 1 & 0 & 0 & \cdots & 0 \\
1 & 0 & mn & 0 & \cdots & 0 \\
2 & 0 & 0 & \beta_{2,2} & \cdots & \beta_{p,2}
\end{array}
\]
where\begin{eqnarray*}
p=\left\{\begin{array}{ll}
 m, & \hbox{if \, $n=1$ ;} \\
2m+n-2, & \hbox{if \, $m\geq n>1.$}
 \end{array}\right.
\end{eqnarray*}
\end{theorem}

\begin{proof} Because of the regularity and depth of $S/J_G$, the non-vanishing
part of the Betti table is concentrated in the frame of the one given in the statement.
Clearly $\beta_{0,0} = 1$ and $\beta_{i,0} = 0$ for all $i > 0$. Furthermore $\beta_{1,0} = \beta_{2,0} =0$.
Since $J_G$ is minimally generated by $mn$ binomials we get that $\beta_{1,1} = mn$ and $\beta_{1,2} = 0$.

In order to prove the statement we have to show that $\beta_{2,1} = 0$ because this implies that
$\beta_{i,1} = 0$ for all $i \geq 2$ as a consequence of the minimality of the free resolution.
Here we have two cases:

\textbf{Case(a):} Let $m\geq n>1$. We take the short exact sequence (1) of Corollary \ref{3.2}. It induces
a graded homomorphism of degree zero
\[
\Tor_2^S(K,S/J_{\tilde{G}}) = K^{b_2(m+n)}(-3) \to \Tor_1^S(K,S/J_G) = K^{mn}(-2).
\]
Therefore it is the zero homomorphism. On the other side it induces a homomorphism
\[
\Tor_3^S(K,S/J_{\tilde{G}}) = K^{b_3(m+n)}(-4) \to
\Tor_2^S(K,S/J_G),
\]
which is the zero homomorphism when restricted to degree 3 since
$\reg S/J_G = 2$. Therefore there is a short exact sequence
of $K$-vector spaces
\[
0 \to \Tor_2^S(K,S/J_G)_3 \to \Tor_2^S(K,S/J_{\tilde{G}}\cap A_n)_3 \oplus
\Tor_2^S(K,S/J_{\tilde{G}} \cap B_m)_3 \to K^{b_2(m+n)} \to 0.
\]
That is $\beta_{2,1}(S/J_G) = \beta_{2,1}(S/J_{\tilde{G}}\cap A_n) + \beta_{2,1}(S/J_{\tilde{G}}\cap B_m) -
b_2(m+n)$.

In the next step we shall compute $\beta_{2,1}(S/J_{\tilde{G}}\cap A_n)$ and
$\beta_{2,1}(S/J_{\tilde{G}}\cap B_m)$. We start with the first of them. To this end we
use the short exact sequence (2) of Corollary \ref{3.2}.
At first we note that
$\beta_{1,2}(S/J_{\tilde{G}}\cap A_n) = 0$ which is true since $J_{\tilde{G}}\cap A_n$ is minimally
generated by quadrics as follows by Lemma \ref{5.3}.
Because of
\[
\beta_{3,0}(S/J_{\tilde{G}}\cap A_n) = \beta_{3,0}(S/J_{\tilde{G}}) = \beta_{2,1}(S/A_n) = 0
\]
we get the following exact sequence of $K$-vector spaces.
\begin{gather*}
0 \to \Tor_3^S(K,S/A_n)_3 \to \Tor_3^S(K,S/(J_{\tilde{G}},A_n))_3 \to \Tor_2^S(K,S/J_{\tilde{G}}\cap A_n)_3 \to \\
\Tor_2^S(K,S/J_{\tilde{G}})_3 \to \Tor_2^S(K,S/(J_{\tilde{G}},A_n))_3 \to 0.
\end{gather*}
By counting vector space dimensions this provides that
\[
\beta_{2,1}(S/J_{\tilde{G}}\cap A_n) = b_2(m+n) + \beta_{3,0}(S/(J_{\tilde{G}},A_n)) -
\beta_{2,1}(S/(J_{\tilde{G}},A_n)) - \textstyle{\binom{2n}{3}}.
\]
Since $(J_{\tilde{G}},A_n) = (J_m,A_n)$, where $J_m = I(n+1,\ldots,n+m)$ we might use Lemma
\ref{5.1} for the calculation of these dimensions. Therefore $\Tor_3^S(K,S/(J_{\tilde{G}},A_n))_3 \cong \Tor_0^S(K,S/J_m)_0^{\binom{2n}{3}}$
and
\[
\Tor_2^S(K,S/(J_{\tilde{G}},A_n))_3 \cong \Tor_2^S(K,S/J_m)_3 \oplus \Tor_1^S(K,S/J_m)_2^{\binom{2n}{1}}.
\]
Therefore $\beta_{3,0}(S/(J_{\tilde{G}},A_n)) = \binom{2n}{3}$ and $\beta_{2,1}(S/(J_{\tilde{G}},A_n)) =
b_2(m) +2n b_1(m)$. Putting these integers together it follows that
\[
\beta_{2,1}(S/J_{\tilde{G}}\cap A_n) = b_2(m+n)-b_2(m)-2n b_1(m).
\]
Interchanging the r\^oles of $m$ and $n$ we derive a corresponding formula for $\beta_{2,1}(S/J_{\tilde{G}} \cap
B_m))$, namely
\[
\beta_{2,1}(S/J_{\tilde{G}}\cap B_m) = b_2(m+n)-b_2(n)-2m b_1(n)
\]
Finally we use both expressions in the above formula in order to confirm that $\beta_{2,1}(S/J_G) $ vanishes.

  \textbf{Case(b):}  Let $m>n=1$. Because $J_{\tilde{G}}\cap A_1=J_G$ we might use exact sequence (2) of
  Corollary \ref{3.2}. Since $(J_{\tilde{G}},A_1) = (I(2,\ldots,n+m),A_1)$ the statements in Lemma \ref{5.1}
  imply that $\Tor_3^S(K,S/(J_{\tilde{G}},A_1))_3 = 0$. Whence there is an exact sequence of $K$-vector spaces
 \[
 0 \to \Tor_2^S(K,S/J_G)_3 \to \Tor_2^S(K,S/J_{\tilde{G}})_3 \oplus \Tor_2^S(K,S/A_1)_3 \to
 \Tor_2^S(K,S/(J_{\tilde{G}},A_1))_3 \to 0.
 \]
 Therefore $\beta_{2,1}(S/J_G) = b_2(m+1) -
\beta_{2,1}(S/(J_{\tilde{G}},A_1))$. Again by the statement of Lemma \ref{5.1} (a) it follows that  $\beta_{2,1}(S/(J_{\tilde{G}},A_1))=b_2(m)+2 b_1(m).$ Finally \[ \beta_{2,1}(S/J_G) = b_2(m+1) -b_2(m)-2 b_1(m)=0,\]
as required.
\end{proof}

As a final feature of the investigations we will describe the explicit values of the Betti numbers
$\beta_{2,i}(S/J_G), 2 \leq i \leq p,$ as they are indicated in Theorem \ref{5.2}.

\begin{theorem} Let $G= K_{m,n}$ denote the complete bipartite graph with $m \geq n \geq 1$.
\begin{itemize}
\item[(a)] The Hilbert function of $S/J_G$ is given by
\begin{gather*} H(S/J_G,t) =
\textstyle{\frac{1}{(1-t)^{m+n+1}}(1+(m+n-1)t) +\frac{1}{(1-t)^{2m}} +\frac{1}{(1-t)^{2n}}} \\
\textstyle{-\frac{1}{(1-t)^{m+1}}(1+(m-1)t)  -\frac{1}{(1-t)^{n+1}}(1+(n-1)t) }.
\end{gather*}
\item[(b)] For the multiplicity $e(S/J_G)$ it follows
\[
e(S/J_G) = \begin{cases} 1, & \textrm{ if  } m > n+1 \textrm{ or } n=1 \textrm{ and } m>2,\\
                                     2m,  & \textrm{ otherwise. }
                 \end{cases}
\]
\item[(c)] Let $n= 1$. Then $\beta_{i,2}(S/J_G) = m\binom{m}{i} - \binom{m}{i+1} - \binom{m+1}{i+1}$ for all
$2 \leq i \leq p = m$. Let $m \geq n > 1$. Then
\begin{gather*}
\beta_{i,2}(S/J_G) = \textstyle{\binom{m+n}{i+2} +\binom{2n}{i+2} + \binom{2m}{i+2} +m \binom{m+2n-1}{i+1} +n \binom{2m+n-1}{i+1}} \\
\textstyle{ - \binom{m+2n}{i+2} - \binom{2m+n}{i+2} - (m+n) \binom{m+n-1}{i+1}  }
\end{gather*}
for all $2 \leq i \leq p = 2m+n-2$.
\end{itemize}
\end{theorem}

\begin{proof} In order to prove (a) we use the short exact sequences of Corollary \ref{3.2}. By the additivity
of the Hilbert series we get the following equalities:
\begin{itemize}
\item[]  $H(S/J_G,t) = H(S/J_{\tilde{G}}\cap A_n,t) + H(S/J_{\tilde{G}}\cap B_n,t) - H(S/J_{\tilde{G}},t),$
\item[] $H(S/J_{\tilde{G}}\cap A_n,t) = H(S/J_{\tilde{G}},t) +H(S/A_n,t) - H(S/(J_{\tilde{G}} A_n),t),$ and
\item[] $H(S/J_{\tilde{G}}\cap B_m,t) = H(S/J_{\tilde{G}},t) +H(S/B_m,t) - H(S/(J_{\tilde{G}} B_m),t)$.
\end{itemize}
Substituting the Hilbert series of the complete graphs $S/J_{\tilde{G}}, S/(J_{\tilde{G}},A_n)$ and $S/(J_{\tilde{G}},B_m)$
as well as the Hilbert series of the polynomial rings $S/A_n, S/B_m$ we get the desired formula in (a). Then (b)
is an easy consequence of (a).

For the proof of (c) we note at first the structure of the finite free resolution of $S/J_G$
\[
0 \to S^{\beta_p}(-p-2) \to \cdots \to S^{\beta_3}(-5) \to S^{\beta_2}(-4) \to S^{\beta_1}(-2) \to S
\]
with $\beta_1 =mn$ $\beta_i = \beta_{i,2}(S/J_G), 2 \leq i \leq p,$ as shown in Theorem \ref{5.2}. By the additivity of the
Hilbert series this provides the following expression
\[
H(S/J_G,t) = \frac{1}{(1-t)^{2m+2n}}(1-\beta_1 t^2+\sum_{i=2}^p (-1)^i\beta_i t^{i+2})
\]
(see also \cite[Exercise 19.14]{Ei}).
Now we use the expression of the Hilbert series $H(S/J_G,t)$ as shown in (a) and compare it with
the one of the minimal free resolution. By some nasty calculations we derive the formulas for
the Betti numbers as given in the statement.
\end{proof}

\textbf{Acknowledgement.} The authors are grateful to Prof. J\"urgen Herzog for drawing their attention
to binomial edge ideals and suggesting their investigation.


\begin{thebibliography}{999}

\bibitem [1]{BS} {\sc M. Brodmann, R. Sharp:} Local Cohomology. An Algebraic
Introduction with Geometric Applications. Cambr. Stud. in Advanced Math., No. 60. Cambridge University Press, (1998).


\bibitem[2]{BH} {\sc W. Bruns, J. Herzog:}  Cohen-Macaulay Rings, Cambridge
University Press, 1993.

\bibitem[3]{Ei} {\sc D. Eisenbud:} Commutative Algebra (with a View Toward Algebraic Geometry).
Springer-Verlag, 1995.

\bibitem[4]{EHH}  {\sc V. Ene, J. Herzog and T. Hibi:} Cohen Macaulay Binomial edge ideals.
Nagoya Math. J. 204 (2011) 57-68.

\bibitem[5]{Go} {\sc S. Goto:} Approximately Cohen-Macaulay Rings. J. Algebra 76,
214-225 (1982)

\bibitem[6]{HKR} {\sc J. Herzog, T. Hibi, F. Hreinsdotir, T. Kahle, J, Rauh:}
Binomial edge ideals and conditional independence statements. Adv. Appl. Math. 45 (2010)  317-333.

\bibitem[7]{M} {\sc H. Matsumura:} Commutative Ring Theory, Cambridge Studies in Advanced Mathematics 8, Cambridge University Press, Cambridge (1986).

\bibitem[8]{Sch1} {\sc P. Schenzel:} On The Use of Local Cohomology in Algebra and Geometry. In: Six Lectures in
Commutative Algebra, Proceed. Summer School on Commutative Algebra at Centre de
Recerca Matem\`{a}tica, (Ed.: J. Elias, J. M. Giral, R. M. Mir\'{o}-Roig, S. Zarzuela), Progr. Math. 166, pp. 241-292, Birkh\"auser, 1998.


\bibitem[9]{Sch3} {\sc P.~Schenzel:}  On the dimension filtration and Cohen-Macaulay filtered modules.
Van Oystaeyen, Freddy (ed.), Commutative algebra and algebraic geometry. Proceedings of the Ferrara meeting in honor of Mario Fiorentini on the occasion of his retirement, Ferrara, Italy. New York, NY: Marcel Dekker. Lect. Notes Pure Appl. Math. 206, 245-264 (1999).


\bibitem[10]{Sch2} {\sc P.~Schenzel:}  On birational
Macaulayfications and Cohen-Macaulay canonical modules. J.
Algebra 275 (2004), 751-770.

\bibitem[11]{St} {\sc R. P. Stanley:}
Combinatorics and commutative algebra. 2nd ed.
Progress in Mathematics (Boston, Mass.). 41. Basel: Birkhuser (1996).

\bibitem[12]{Vi}  {\sc R. H. Villarreal:} Monomial Algebras, New York: Marcel Dekker Inc. (2001).

\bibitem[13]{So} {\sc S. Zafar:} On approximately Cohen-Macaulay binomial edge ideal, to appear in Bull. Math. Soc. Sci. Math. Roumanie.

\end{thebibliography}
\end{document}